\documentclass[11pt]{article}

\usepackage[margin=1.12in]{geometry}
\usepackage{amsmath,amssymb,amsthm,mathtools,mathrsfs}
\usepackage{bm}
\usepackage{microtype}
\usepackage{enumitem}
\usepackage{booktabs}
\usepackage{lmodern}
\usepackage[T1]{fontenc}
\DeclareUnicodeCharacter{FF0C}{,}
\usepackage{hyperref}
\usepackage{xcolor}

\newtheorem{theorem}{Theorem}[section]
\newtheorem{proposition}[theorem]{Proposition}
\newtheorem{lemma}[theorem]{Lemma}
\newtheorem{corollary}[theorem]{Corollary}
\newtheorem{remark}[theorem]{Remark}

\theoremstyle{definition}

\newcommand{\R}{\mathbb{R}}
\newcommand{\E}{\mathcal{E}}

\newcommand{\dd}{\,\mathrm{d}}
\newcommand{\sgn}{\operatorname{sgn}}

\title{\bf Propagation Direction of Bistable Traveling Fronts in the
Lotka--Volterra Competition--Diffusion System}
\author{
    Shizhao Ma\textsuperscript{1,2},
    Dongyuan Xiao\textsuperscript{1},
    Maolin Zhou\textsuperscript{3}
    \\[0.8em]
    \small
    \textsuperscript{1}School of Mathematical Sciences,
    Shanghai Jiao Tong University,
    Shanghai 200240, P.R. China
    \\
    \small
    \textsuperscript{2}Institute of Natural Sciences,
    Shanghai Jiao Tong University,
    Shanghai 200240, P.R. China
    \\
    \small
    \textsuperscript{3}Chern Institute of Mathematics, Nankai University, Tianjin 300071, P.R. China
}

\date{}

\begin{document}
\maketitle
\medskip

\begin{abstract}
We study the propagation direction of bistable traveling fronts in the
two-species Lotka--Volterra competition--diffusion system under strong
competition. A complete characterization of the sign of the wave speed
has remained a long-standing unsolved problem. 
We establish the first global
necessary and sufficient criterion for zero wave speed by combining a
Maxwell-type identity with a phase-plane rigidity argument. This criterion
yields a unique zero-speed threshold, identifies its value in the symmetric
case and its limiting values in two extreme regimes, and consequently
determines the propagation direction throughout the strong-competition
parameter region.
\end{abstract}

% \begin{abstract}
% We study the propagation direction of bistable traveling fronts in two-species Lotka--Volterra competition--diffusion system in the strong-competition regime. Although the propagation speed is known to be strictly monotone with respect to the interspecific competition coefficients, its dependence on the diffusion and intrinsic growth parameters is substantially more delicate. In this paper, we first show that every monotone standing front satisfies an exact Maxwell-type identity in which the parameter imbalance is compensated by a nonnegative geometric defect of the projected heteroclinic orbit. Second, for fixed competition coefficients, we prove that at most one diffusion-to-growth ratio can support a monotone standing front.  Combining this with the classical existence, uniqueness, continuity, and parameter monotonicity of bistable fronts yields a unique threshold
% which is continuous and strictly decreasing in the diffusion-to-growth ratio.
% Consequently, the propagation direction can be completely characterized throughout the strong-competition parameter region.
% \end{abstract}

\medskip
\noindent{\em MSC 2020:} 35C07, 35K57, 34C37, 92D25.\\
{\em Keywords:} Lotka--Volterra competition-diffusion system; bistable traveling wave;
% propagation direction; 
%standing wave; 
phase plane analysis; heteroclinic orbit.
%; Maxwell identity.

\section{Introduction and main results}

We study the one-dimensional Lotka--Volterra competition--diffusion system
\begin{equation}\label{PDE}
\begin{cases}
u_t=u_{xx}+u(1-u-k_1v),\\[1mm]
v_t=dv_{xx}+rv(1-k_2u-v),
\end{cases}
\qquad t>0,\quad x\in\mathbb R.
\end{equation}
Here \(u=u(t,x)\) and \(v=v(t,x)\) denote the population densities of the two competing species. We normalize both the diffusion coefficient and the intrinsic growth rate of the \(u\)-species to one. Thus, \(d>0\) and \(r>0\) are, respectively, the diffusion coefficient and intrinsic growth rate of the \(v\)-species relative to those of the \(u\)-species. The coefficients \(k_1>0\) and \(k_2>0\) measure the competitive effects exerted by the \(v\)-species on the \(u\)-species and by the \(u\)-species on the \(v\)-species, respectively.

Before turning to the strong competition system, it is
useful to recall its scalar limits.  If the \(v\)-species is absent,
then \(v\equiv0\) and the \(u\)-component satisfies
\[
u_t=u_{xx}+u(1-u).
\]
Similarly, if the \(u\)-species is absent, then \(u\equiv0\) and the
\(v\)-component satisfies
\[
v_t=dv_{xx}+rv(1-v).
\]
Both equations belong to the classical Fisher--KPP family introduced
independently by Fisher \cite{Fisher1937} and by Kolmogorov,
Petrovskii and Piskunov \cite{KPP1937}. More generally, the scalar
equation
\[
w_t=Dw_{xx}+Rw(1-w),
\qquad D>0,\quad R>0,
\]
admits an increasing traveling profile
\[
w(t,x)=W(x+ct),
\qquad
W(-\infty)=0,\quad W(+\infty)=1,
\]
if and only if
\[
c\geq 2\sqrt{DR}.
\]
Hence the uncoupled \(u\)- and \(v\)-equations have the Fisher--KPP minimal wave speeds \(c_u=2\) and \(c_v=2\sqrt{dr}\), respectively.

Throughout the paper, we assume
\begin{equation}\label{SC}
k_1>1\ \ \text{and}\ \ k_2>1.
\end{equation}
Under \eqref{SC}, the two semi-trivial equilibria
\((1,0)\) and \((0,1)\) are stable for the associated kinetic system, and the traveling wave is of bistable type. We seek monotone bistable fronts of the form
\[
(u,v)(t,x)=(U,V)(\xi),\qquad \xi=x+ct,
\]
where \(c\in\mathbb R\) is the propagation speed and \(U,V\) are the wave profiles. They satisfy
\begin{equation}\label{TW}
\begin{cases}
U''-cU'+U(1-U-k_1V)=0,\\
dV''-cV'+rV(1-k_2U-V)=0,
\end{cases}
\qquad \xi\in\mathbb R,
\end{equation}
with
\begin{equation}\label{BC}
(U,V)(-\infty)=(0,1),\qquad
(U,V)(+\infty)=(1,0),
\end{equation}
and
\begin{equation}\label{mono}
U'>0>V'.
\end{equation}
We denote by \(c=c(d,r,k_1,k_2)\) the unique propagation speed. 
Here and below, primes denote differentiation with respect to the traveling wave variable \(\xi\). 
In view of the end states in \eqref{BC}, \(c>0\) means that the
\(u\)-species invades the region occupied by the \(v\)-species, whereas
\(c<0\) means that the \(v\)-species invades the region occupied by the
\(u\)-species.

The existence and uniqueness, up to translation, of the monotone bistable front and the uniqueness of its propagation speed are classical; see Gardner \cite{Gardner1982}, Kan-on \cite{Kanon1995}, and the recent work of Nakamura and Ogiwara \cite{NakamuraOgiwara2026}. Moreover, the speed is strictly decreasing in \(k_1\) and strictly increasing in \(k_2\) \cite{Kanon1995,NakamuraOgiwara2026}. Consequently, for each fixed \(d,r,k_2\), there exists a unique threshold
\[
k_*=k_*(d,r,k_2)>1
\]
such that
\begin{equation}\label{thresholdknown}
c>0\iff k_1<k_*,\qquad
c=0\iff k_1=k_*,\qquad
c<0\iff k_1>k_*.
\end{equation}

The sign of the bistable wave speed has been investigated from several perspectives. Standing wave
criteria and parameter-dependent sign conditions were obtained by
Kan-on \cite{Kanon1996} and Guo and Lin \cite{GuoLin2013}.
Chang, Chen, and Wang \cite{ChangChenWang2023} introduced a minimax
characterization of the zero-speed condition and derived further
explicit sign criteria. More recently, Xiao \cite{Xiao2025} obtained
sufficient conditions for the propagation direction by analyzing an
associated degenerate competition system. These results provide important partial descriptions of the parameter space, but they do not establish global uniqueness and monotonicity of the zero-speed curve \(c(d,r,k_1,k_2)=0\) as a function of the diffusion-to-growth ratio \(\rho=d/r\).

% The more delicate issue is the dependence of the propagation direction on the diffusion and intrinsic growth parameters \(d\) and \(r\). Recent work explicitly records the corresponding monotonicity problem as open in general \cite{NakamuraOgiwara2026}. 
% Numerical computations in \cite{AlzahraniDavidsonDodds2012}, summarized in Remark~5.2 of \cite{NakamuraOgiwara2026}, suggest that, for fixed \(r\) and \(k_2\), the zero-speed curve in the \((k_1,d)\)-plane is monotone, passes through \((k_2,r)\), and has limiting points \((k_2^2,0)\) and \((\sqrt{k_2},\infty)\). 
% Nakamura and Ogiwara \cite{NakamuraOgiwara2026} further noted that a complete proof of this picture, including the exact universal thresholds, had not yet been established.

The more delicate issue is the dependence of the propagation
direction on the diffusion and intrinsic growth parameters \(d\) and
\(r\). Recent work explicitly records the corresponding monotonicity
problem as open in general \cite{NakamuraOgiwara2026}. Numerical
computations in \cite{AlzahraniDavidsonDodds2012}, summarized in
Remark~5.2 of \cite{NakamuraOgiwara2026}, suggest that, for fixed
\(r\) and \(k_2\), the zero-speed curve $c(d,k_1)=0$ in the \((k_1,d)\)-plane is
monotone and satisfies
\[
k_1\longrightarrow k_2^2
\qquad\text{as }d\downarrow0,
\]
and
\[
k_1\longrightarrow\sqrt{k_2}
\qquad\text{as }d\to\infty.
\]
Nakamura and Ogiwara \cite{NakamuraOgiwara2026} further noted that a
complete proof of this picture, including the exact universal
thresholds, had not yet been established.

The purpose of this paper is to complete this picture. The key observation is that the standing wave problem has substantially more structure than the general traveling wave problem. Setting \(c=0\) in \eqref{TW} and dividing the second equation by \(r\), we obtain
\begin{equation}\label{SW}
\begin{cases}
U''+U(1-U-k_1V)=0,\\
\rho V''+V(1-k_2U-V)=0,
\end{cases}
\qquad \xi\in\mathbb R,
\end{equation}
together with
\[
(U,V)(-\infty)=(0,1),\qquad
(U,V)(+\infty)=(1,0),
\qquad
U'>0>V'.
\]
Thus, at zero speed, the diffusion coefficient \(d\) and intrinsic growth rate \(r\) enter the standing wave system only through the ratio \(\rho=d/r\).

Our first result is a Maxwell-type identity.

\begin{theorem}[Zero-speed Maxwell identity]\label{thm:Maxwell}
Let $k_1,k_2>1$ and $\rho>0$.  If \eqref{SW} admits a monotone standing front, then
\begin{equation}\label{Maxwell}
 3k_1k_2(\rho-1)\E+(k_1-k_2)\bigl[\rho(k_1-1)+(k_2-1)\bigr]=0,
\end{equation}
where
\begin{equation}\label{Edef}
 \E:=\int_{\R}(U+V-1)^2U'(\xi)\,\dd \xi>0.
\end{equation}
Consequently,
\begin{equation}\label{oppositesign}
 (\rho-1)(k_1-k_2)<0
\end{equation}
unless $\rho=1$ and $k_1=k_2$.
\end{theorem}

Our second result establishes rigidity with respect to the diffusion-to-growth ratio.

\begin{theorem}\label{thm:rigidity}
Fix $k_1,k_2>1$.  There exists at most one $\rho>0$ for which \eqref{SW} admits a monotone standing front.
\end{theorem}

The proof of Theorem \ref{thm:rigidity} does not use any monotonicity of the wave speed with respect to \(d\) or \(r\). We compare two hypothetical standing phase curves \(V=H_i(U)\) corresponding to \(\rho_1<\rho_2\). Two exact moment identities, together with the opposite endpoint orderings of the curves, lead to a lens-shaped graph formed by the two phase curves. A discrete summation argument then indicates a downward crossing at which the associated kinetic quantities force the opposite slope ordering, yielding a contradiction. The argument is developed in Section \ref{sec:lens}.

We can now state the main threshold theorem. Its proof also uses the standard continuous dependence of the unique bistable front and its speed on the coefficients (see \cite{Kanon1995,Volpert1994}).

% We now state the global consequence. The external traveling wave
% results used below are summarized more precisely in
% Proposition~\ref{prop:standard_front_dependence}. For each fixed
% \(d>0\), the \(C^1\)-dependence of the monotone wave profiles and
% their speed on the kinetic parameters, together with uniqueness up
% to translation and the strict parameter monotonicities, follows from
% Kan-on \cite[Theorem~2.1]{Kanon1995}. The general existence and
% uniqueness theory for monotone bistable systems is given in
% \cite[Chapter~3, Theorem~1.1]{VolpertEtAl1994}, while the uniform
% profile and speed estimates required in the compactness argument are
% provided by
% \cite[Chapter~3, Proposition~1.2 and Theorem~2.1]{VolpertEtAl1994}.
% Only continuity in \(\rho\) is needed here, not monotonicity of the
% speed with respect to \(d\) or \(r\).

\begin{theorem}\label{thm:main}
Fix $k_2>1$.  There exists a unique continuous function
\[
 K_{k_2}:(0,\infty)\longrightarrow(1,\infty)
\]
such that, for every $d,r>0$ and $k_1>1$ with $\rho=d/r$,
\begin{equation}\label{Kzero}
 c(d,r,k_1,k_2)=0\quad\Longleftrightarrow\quad k_1=K_{k_2}(\rho).
\end{equation}
Moreover:
\begin{enumerate}[label=\rm(\roman*)]
\item $K_{k_2}$ is strictly decreasing on $(0,\infty)$;
\item
\begin{equation}\label{threepoints}
 K_{k_2}(1)=k_2,\qquad
 \lim_{\rho\downarrow0}K_{k_2}(\rho)=k_2^2,
 \qquad
 \lim_{\rho\to\infty}K_{k_2}(\rho)=\sqrt k_2;
\end{equation}
\item the speed is classified by
\begin{equation}\label{signK}
 \begin{aligned}
 c(d,r,k_1,k_2)>0&\iff k_1<K_{k_2}(\rho),\\
 c(d,r,k_1,k_2)=0&\iff k_1=K_{k_2}(\rho),\\
 c(d,r,k_1,k_2)<0&\iff k_1>K_{k_2}(\rho).
 \end{aligned}
\end{equation}
\end{enumerate}
In particular, $K_{k_2}$ maps $(0,\infty)$ bijectively onto $(\sqrt k_2,k_2^2)$.
\end{theorem}

The limit \(k_2^2\) at \(\rho\downarrow0\) follows from the rigorous near-degenerate threshold obtained by Alzahrani--Davidson--Dodds \cite{AlzahraniDavidsonDodds2010}; see also Section~5 of \cite{NakamuraOgiwara2026}. The opposite limit \(\sqrt{k_2}\) follows from the exact species-exchange identity.
\begin{equation}\label{exchange-intro}
 c(d,r,k_1,k_2)=-\sqrt{dr}\;c(1/d,1/r,k_2,k_1),
\end{equation}
combined with the $d/r\downarrow0$ limit.

Two useful consequences are immediate.

\begin{corollary}
\label{cor:sharp}
Fix \(k_2>1\).

\begin{enumerate}
\item[(i)]
If \(k_1\geq k_2^2\), then
\[
c(d,r,k_1,k_2)<0
\]
for every \(d,r>0\).

\item[(ii)]
If \(1<k_1\leq \sqrt{k_2}\), then
\[
c(d,r,k_1,k_2)>0
\]
for every \(d,r>0\).

\item[(iii)]
If
\[
\sqrt{k_2}<k_1<k_2^2,
\]
then there exists a unique
\[
\rho_*=\rho_*(k_1,k_2)>0
\]
such that
\[
c(d,r,k_1,k_2)=0
\iff
\frac{d}{r}=\rho_*.
\]
Moreover,
\[
\frac{d}{r}<\rho_*
\implies
c(d,r,k_1,k_2)>0,
\qquad
\frac{d}{r}>\rho_*
\implies
c(d,r,k_1,k_2)<0.
\]
\end{enumerate}
\end{corollary}

% \begin{corollary}[Sharp universal competition thresholds]\label{cor:sharp}
% Fix $k_2>1$.
% \begin{enumerate}[label=\rm(\roman*)]
% \item If $k_1\ge k_2^2$, then $c(d,r,k_1,k_2)<0$ for every $d,r>0$.
% \item If $1<k_1\le\sqrt{k_2}$, then $c(d,r,k_1,k_2)>0$ for every $d,r>0$.
% \item If $\sqrt{k_2}<k_1<k_2^2$, there is a unique $\rho_*=\rho_*(k_1,k_2)>0$ such that
% \[
%  c=0\iff d/r=\rho_*,
% \]
% and
% \[
%  \frac dr<\rho_*\Longrightarrow c>0,
%  \qquad
%  \frac dr>\rho_*\Longrightarrow c<0.
% \]
% \end{enumerate}
% \end{corollary}

\begin{corollary}[Symmetric competition]\label{cor:symmetric}
If $k_1=k_2=k>1$, then
\begin{equation}\label{symzero}
 c(d,r,k,k)=0\quad\Longleftrightarrow\quad d=r,
\end{equation}
and
\begin{equation}\label{symsign}
 \sgn c(d,r,k,k)=\sgn(r-d).
\end{equation}
In particular, when $r=1$, the faster diffuser has the propagation advantage in the sense that $d>1$ implies $c<0$ in the convention \eqref{TW}.
\end{corollary}

\begin{remark}
In the critical case \(k=1\), the result referred to as \cite{Alf-Xiao} shows that the species with the larger propagation speed has the propagation advantage. In particular, when \(r=1\), this again favors the faster diffuser, consistently with the conclusion above for the strong-competition regime.
\end{remark}

A recent preprint by Chen and Wang
\cite{ChenWang2026} establishes this propagation direction near the
strong-competition borderline \(k\downarrow1\), while also noting that
the full symmetric conjecture remains beyond the scope of their
method. Corollary~\ref{cor:symmetric} resolves the full
symmetric case through zero-speed phase-plane analysis rather than a
direct comparison of nonzero wave speeds.

The remainder of the paper is organized as follows. Section~2 derives two universal moment identities for standing fronts and proves the Maxwell-type identity. Section~3 establishes the endpoint asymptotics and the corresponding ordering of standing phase curves. Section~4 develops the phase plane analysis and proves the fixed-competition zero-speed rigidity theorem. Section~5 combines this rigidity with classical traveling wave theory to characterize the global zero-speed threshold. 

\section{The Maxwell identity for standing fronts}\label{sec:identities}

Throughout Sections \ref{sec:identities}--\ref{sec:lens}, we fix \(k_1,k_2>1\), and \((U,V)\) denotes a monotone solution of \eqref{SW}. Standard hyperbolicity of the two end equilibria implies exponential convergence of the profile. In particular,
\begin{equation}\label{derzero}
 U'(\pm\infty)=V'(\pm\infty)=0.
\end{equation}
All integrations by parts below are therefore justified.  Alternatively, \eqref{derzero} follows from monotonicity, $U',-V'\in L^1(\R)$, and boundedness of $U'',V''$.

Since $U'>0$, the projected phase orbit is a graph
\begin{equation}\label{Hgraph}
 V=H(U),\qquad H:(0,1)\to(0,1),\qquad H'<0,
\end{equation}
with $H(0)=1$ and $H(1)=0$ in the endpoint sense.

\bigskip
We first establish two universal moment identities for the projected phase orbit \(H\).
\begin{lemma}\label{lem:moments}
Every monotone standing front satisfies
\begin{equation}\label{moment1}
 \int_0^1uH(u)\,\dd u=\frac{1}{6k_1},
\end{equation}
and
\begin{equation}\label{moment2}
 \int_0^1H(u)^2\,\dd u=\frac{1}{3k_2}.
\end{equation}
\end{lemma}

\begin{proof}
Multiply the first equation of \eqref{SW} by $U'$ and integrate over $\R$.  The derivative term vanishes by \eqref{derzero}, hence
\[
 0=\int_\R U(1-U-k_1V)U'\,\dd \xi
 =\frac{1}{6}-k_1\int_\R UVU'\,\dd \xi,
\]
which gives \eqref{moment1}.  Similarly, multiplying the second equation by $V'$ gives
\[
 0=\int_\R V(1-k_2U-V)V'\,\dd \xi
 =-\frac{1}{6}-k_2\int_\R UVV'\,\dd \xi.
\]
Since
\[
 \int_\R UVV'\,\dd \xi
 =-\frac{1}{2}\int_\R U'V^2\,\dd \xi,
\]
we obtain
\[
 \int_\R U'V^2\,\dd \xi=\frac{1}{3k_2},
\]
which is \eqref{moment2} after the change of variable $u=U(\xi)$.
\end{proof}

Identities of this type are closely related to the standing wave moment calculations of Guo--Lin \cite{GuoLin2013}. We now prove the Maxwell identity.

\begin{proof}[Proof of Theorem \ref{thm:Maxwell}]
Set
\[
 M:=\int_\R U'V\,\dd \xi=\int_0^1H(u)\,\dd u.
\]
Multiply the first equation of \eqref{SW} by $\rho V'$, the second by $U'$, and add.  Since
\[
 \rho(U''V'+U'V'')=\rho(U'V')',
\]
integration gives
\begin{equation}\label{cross0}
 \rho I_1+I_2=0,
\end{equation}
where
\[
 I_1:=\int_\R U(1-U-k_1V)V'\,\dd \xi,
 \qquad
 I_2:=\int_\R V(1-k_2U-V)U'\,\dd \xi.
\]
By integration by parts and Lemma \ref{lem:moments},
\begin{align}
 I_1
 &= -M+2\int_\R UU'V\,\dd \xi
      +\frac{k_1}{2}\int_\R U'V^2\,\dd \xi \notag\\
 &= -M+\frac{1}{3k_1}+\frac{k_1}{6k_2},\label{I1}
\end{align}
and
\begin{equation}\label{I2}
 I_2=M-\frac{1}{3k_2}-\frac{k_2}{6k_1}.
\end{equation}
On the other hand,
\begin{align}
 \E
 &=\int_\R(U+V-1)^2U'\,\dd \xi\notag\\
 &=\frac{1}{3}-2M+2\int_\R UVU'\,\dd \xi+\int_\R U'V^2\,\dd \xi\notag\\
 &=\frac{1}{3}-2M+\frac{1}{3k_1}+\frac{1}{3k_2}.
 \label{E-M}
\end{align}
Thus
\begin{equation}\label{M-E}
 M=\frac{1}{6}+\frac{1}{6k_1}+\frac{1}{6k_2}-\frac{\E}{2}.
\end{equation}
Substituting \eqref{M-E} into \eqref{cross0}--\eqref{I2} and simplifying yields exactly \eqref{Maxwell}.

By definition, \(\E\geq0\). If \(\E=0\), then \(U+V\equiv1\) because \(U'>0\). 
Substituting $V = 1-U$ % Substitution 
into the first equation of \eqref{SW}
yields % gives
\[
 U''-(k_1-1)U(1-U)=0.
\]
Multiplying by $U'$ and integrating gives
\[
 0=-(k_1-1)\int_0^1u(1-u)\,\dd u=-\frac{k_1-1}{6},
\]
a contradiction.  Hence $\E>0$.  Since
\[
 \rho(k_1-1)+(k_2-1)>0,
\]
\eqref{oppositesign} follows from \eqref{Maxwell}.
\end{proof}

\begin{remark}
The quantity
\[
\E=\int_0^1(H(u)-(1-u))^2\,\dd u
\]
measures the squared \(L^2(0,1)\)-distance between the projected standing orbit
\(V=H(U)\) and the anti-diagonal \(V=1-U\). Therefore, the Maxwell identity (8)
admits a geometric interpretation: the asymmetry of the model parameters is
balanced by the deviation of the heteroclinic orbit from the symmetric configuration.

More precisely, the second term in (8) represents the imbalance caused by the
competition parameters and the diffusion-to-growth ratio, while the first term
contains the nonnegative geometric defect \(\E\) of the standing orbit. Thus, the
parameter mismatch can only be compensated through the geometric deformation of
the phase curve.

Although \(\E\) has a natural geometric interpretation, it depends on the standing
orbit itself and therefore (8) alone does not yield the uniqueness of the
diffusion-to-growth ratio. In the symmetric case \(k_1=k_2\), since \(\E>0\),
identity (8) immediately gives \(\rho=1\).
\end{remark}

% \begin{remark}[Geometric interpretation]\label{rem:geometry}
% The quantity
% \[
%  \E=\int_0^1\bigl(u+H(u)-1\bigr)^2\,\dd u
% \]
% is the squared $L^2(\dd u)$ distance of the projected standing orbit from the anti-diagonal $V=1-U$.  Thus \eqref{Maxwell} is a system analogue of the scalar Maxwell balance: a parameter imbalance is compensated by a nonnegative geometric defect of the heteroclinic orbit.  When $k_1=k_2$, the parameter term vanishes and \eqref{Maxwell} immediately forces $\rho=1$.
% \end{remark}

\section{Endpoint asymptotics and ordering of standing phase curves}\label{sec:endpoints}

Our phase plane analysis requires a precise ordering of the standing phase curves near both endpoints. This ordering follows from the stable- and unstable-manifold expansions at the hyperbolic equilibria. We record the leading-order asymptotics explicitly because the resonant cases are needed in the comparison.

\begin{lemma}[Left endpoint]\label{lem:left}
Let $H_\rho$ be a standing phase curve for fixed $k_1,k_2>1$.  As $u\downarrow0$,
\begin{equation}\label{leftasympt}
1-H_\rho(u)
\sim
\begin{cases}
\displaystyle
\frac{k_2}{1-\rho(k_1-1)}\,u,
&
\rho(k_1-1)<1,
\\[3mm]
\displaystyle
\frac{k_2}{2}\,u|\log u|,
&
\rho(k_1-1)=1,
\\[3mm]
\displaystyle
C_-(\rho)\,
u^{1/\sqrt{\rho(k_1-1)}},
&
\rho(k_1-1)>1,
\end{cases}
\end{equation}
where $C_-(\rho)>0$ in the last case.  Consequently, if $0<\rho_1<\rho_2$ and both standing curves exist for the same $k_1,k_2$, then
\begin{equation}\label{leftorder}
 H_{\rho_2}(u)<H_{\rho_1}(u)
\end{equation}
for all sufficiently small $u>0$.
\end{lemma}

\begin{proof}
Set $W=1-V$.  Near $(U,W)=(0,0)$, \eqref{SW} has linearization
\begin{equation}\label{leftlin}
 U''-(k_1-1)U=0,
 \qquad
 \rho W''-W+k_2U=0.
\end{equation}
Let $\lambda=\sqrt{k_1-1}$.  Along the heteroclinic,
\[
U(\xi)
=
Ae^{\lambda\xi}(1+o(1)),
\qquad
A>0,
\qquad
\xi\to-\infty.
\]
% The homogeneous exponent for the $W$-equation is $\mu=1/\sqrt\rho$.
The positive characteristic exponent of the homogeneous
\(W\)-equation is
$
\mu=\frac{1}{\sqrt{\rho}}.
$
If \(\lambda<\mu\), the forced \(U\)-mode dominates, and\[
 W=\frac{k_2}{1-\rho(k_1-1)}U+o(U).
\]
% If $\lambda=\mu$, the resonant particular solution is proportional to $\xi e^{\lambda \xi}$ and yields
% \[
%  W=\frac b2U|\log U|+O(U).
% \]
If \(\lambda=\mu\), resonance produces a particular solution
proportional to \(\xi e^{\lambda\xi}\), yielding
\[
W
=
\frac{k_2}{2}\,U|\log U|+O(U).
\]
If $\lambda>\mu$, the slower homogeneous $W$-mode dominates, and
\[
 W=C e^{\mu \xi}(1+o(1))=C_-(\rho)U^{\mu/\lambda}(1+o(1)).
\]
% The coefficient must be positive: if it vanished, the forced $U$-mode would have coefficient $k_2/[1-\rho(k_1-1)]<0$, contradicting $W=1-V>0$ near $-\infty$.
The coefficient \(C_-(\rho)\) must be positive. Indeed, if the
homogeneous mode vanished, then the forced \(U\)-mode would have the
coefficient
$
\frac{k_2}{1-\rho(k_1-1)}<0,
$
contradicting \(W=1-V>0\) near \(-\infty\).

% The ordering follows directly.  In the first regime the linear coefficient is strictly increasing in $\rho$; in the third regime the exponent $1/\sqrt{\rho(k_1-1)}$ is strictly decreasing in $\rho$, so exponent comparison dominates any positive leading coefficient.  The resonant term lies strictly between the neighboring orders because, as $u\downarrow0$,
% \[
%  u^\theta\gg u|\log u|\gg u\qquad(0<\theta<1).
% \]
% Thus $1-H_{\rho_2}>1-H_{\rho_1}$ near $0$, proving \eqref{leftorder}.
The claimed ordering follows from these asymptotics. In the first regime, the coefficient
$
\frac{k_2}{1-\rho(k_1-1)}
$
is strictly increasing in \(\rho\). In the third regime, the exponent
$
\frac{1}{\sqrt{\rho(k_1-1)}}
$
is strictly decreasing in \(\rho\), and the comparison of powers
dominates the positive leading coefficients. The resonant scale lies
strictly between the two neighboring orders because
\[
u^\theta\gg u|\log u|\gg u
\qquad
\text{as }u\downarrow0,
\qquad
0<\theta<1.
\]
Therefore,
\[
1-H_{\rho_2}(u)>1-H_{\rho_1}(u)
\]
near \(u=0\), which proves \eqref{leftorder}.
\end{proof}

% \begin{lemma}[Right endpoint]\label{lem:right}
% As $u\uparrow1$,
% \begin{equation}\label{rightasympt}
%  H_\rho(u)\sim
%  \begin{cases}
%  C_+(\rho)(1-u)^{\sqrt{(k_2-1)/\rho}},
%      &\rho<k_2-1,\\[2mm]
%  \displaystyle \frac{2}{k_1}\frac{1-u}{|\log(1-u)|},
%      &\rho=k_2-1,\\[3mm]
%  \displaystyle \frac{\rho-(k_2-1)}{a\rho}(1-u),
%      &\rho>k_2-1,
%  \end{cases}
% \end{equation}
% where $C_+(\rho)>0$ in the first case.  Consequently, if $0<\rho_1<\rho_2$, then
% \begin{equation}\label{rightorder}
%  H_{\rho_2}(u)>H_{\rho_1}(u)
% \end{equation}
% for all $u<1$ sufficiently close to $1$.
% \end{lemma}

\begin{lemma}[Right endpoint]\label{lem:right}
As \(u\uparrow1\),
\begin{equation}\label{rightasympt}
H_\rho(u)
\sim
\begin{cases}
\displaystyle
C_+(\rho)
(1-u)^{\sqrt{(k_2-1)/\rho}},
&
\rho<k_2-1,
\\[3mm]
\displaystyle
\frac{2}{k_1}\,
\frac{1-u}{|\log(1-u)|},
&
\rho=k_2-1,
\\[4mm]
\displaystyle
\frac{\rho-(k_2-1)}{k_1\rho}\,(1-u),
&
\rho>k_2-1,
\end{cases}
\end{equation}
where \(C_+(\rho)>0\) in the first case. Consequently, if
\(0<\rho_1<\rho_2\), then
\begin{equation}\label{right_endpoint_ordering}
H_{\rho_2}(u)>H_{\rho_1}(u)
\end{equation}
for all \(u<1\) sufficiently close to \(1\).
\end{lemma}

\begin{proof}
Set
\[
\eta=1-U,
\qquad
\zeta=V.
\]
Near \((\eta,\zeta)=(0,0)\), system \eqref{SW} has the linearization
\begin{equation}\label{right_linearization}
\eta''-\eta+k_1\zeta=0,
\qquad
\rho\zeta''-(k_2-1)\zeta=0.
\end{equation}
The decay exponent of \(\zeta\) is
\[
\mu=\sqrt{\frac{k_2-1}{\rho}},
\]
whereas the homogeneous \(\eta\)-mode has exponent \(1\).

If \(\mu>1\), the homogeneous \(\eta\)-mode decays more slowly, so
\[
\zeta\asymp\eta^\mu.
\]
If \(\mu=1\), resonance yields
\[
\zeta
\sim
\frac{2}{k_1}\,
\frac{\eta}{|\log\eta|}.
\]
If \(\mu<1\), the slower \(\zeta\)-mode forces
\[
\eta
\sim
\frac{k_1\rho}{\rho-(k_2-1)}\,\zeta,
\]
or equivalently,
\[
\zeta
\sim
\frac{\rho-(k_2-1)}{k_1\rho}\,\eta.
\]
This gives \eqref{rightasympt}. The coefficient
\(C_+(\rho)\) is positive because \(V>0\).

The ordering is obtained as in the proof of Lemma \ref{lem:left}. In the first regime, the exponent
\[
\sqrt{\frac{k_2-1}{\rho}}
\]
is strictly decreasing in \(\rho\), while in the third regime the
coefficient
\[
\frac{1}{k_1}
\left(
1-\frac{k_2-1}{\rho}
\right)
\]
is strictly increasing in \(\rho\). The resonant scale lies strictly
between the two neighboring orders. This proves
\eqref{right_endpoint_ordering}.
\end{proof}

% \begin{proof}
% Set $\eta=1-U$ and $\zeta=V$.  Near $(\eta,\zeta)=(0,0)$,
% \begin{equation}\label{rightlin}
%  \eta''-\eta+k_1q=0,
%  \qquad
%  \rho \zeta''-(k_2-1)\zeta=0.
% \end{equation}
% The decay exponent for $\zeta$ is $\mu=\sqrt{(k_2-1)/\rho}$, while the homogeneous $\eta$-mode has exponent $1$.  If $\mu>1$, the $\eta$-mode is slower, so $\zeta\asymp \eta^\mu$.  If $\mu=1$, resonance gives $\zeta\sim(2/k_1)\eta/|\log \eta|$.  If $\mu<1$, the slower $\zeta$-mode forces
% \[
%  \eta\sim \frac{a\rho}{\rho-(k_2-1)}\zeta,
% \]
% which gives the last line of \eqref{rightasympt}.  Positivity of $C_+$ follows from $V>0$.  The ordering is obtained as in Lemma \ref{lem:left}: the exponent decreases with $\rho$ in the first regime, and the linear coefficient
% \[
%  \frac{1}{k_1}\left(1-\frac{k_2-1}{\rho}\right)
% \]
% is strictly increasing in the third regime.  The resonant scale lies strictly between the two neighboring orders.
% \end{proof}

\section{Rigidity of the diffusion-to-growth ratio for standing fronts}\label{sec:lens}

We now prove Theorem \ref{thm:rigidity}. The argument uses only the phase graphs, the endpoint orderings, and the two moment identities obtained in Lemma \ref{lem:moments}.

Assume for contradiction that the same $k_1,k_2>1$ admit standing fronts at
\[
 0<\rho_1<\rho_2.
\]
Let $H_i$ denote the corresponding phase graphs and set
\begin{equation}\label{DG}
 D(u):=H_2(u)-H_1(u),
 \qquad
 G(u):=H_1(u)+H_2(u).
\end{equation}
By Lemmas \ref{lem:left}--\ref{lem:right},
\begin{equation}\label{endsigns}
 D(u)<0\quad(u\downarrow0),
 \qquad
 D(u)>0\quad(u\uparrow1).
\end{equation}
Moreover, Lemma \ref{lem:moments} gives the two zero moments
\begin{equation}\label{Dmoments}
 \int_0^1uD(u)\,\dd u=0,
 \qquad
 \int_0^1G(u)D(u)\,\dd u=0.
\end{equation}

We first show that the graph of $D$ can change sign for only finite times in \((0,1)\).
\begin{lemma}\label{lem:finite}
The set of zeros of $D$ in $(0,1)$ is finite unless $D\equiv0$.  The latter alternative is impossible when $\rho_1\ne\rho_2$.
\end{lemma}

\begin{proof}
The ODE system \eqref{SW} is analytic.  Since $U_i'>0$, the analytic inverse function theorem shows that each $H_i$ is real analytic on $(0,1)$; hence so is $D$.  If $D\not\equiv0$, its zeros are isolated.  By \eqref{endsigns}, no zeros occur in fixed neighborhoods of $0$ and $1$, so all zeros lie in a compact subinterval and are finite.

% If $D\equiv0$, write the common graph as $H$.  Define
% \begin{equation}\label{Pdef}
%  P(u):=(U')^2
%  =-2\int_0^u s\bigl(1-s-aH(s)\bigr)\,\dd s.
% \end{equation}
% Thus $P$ is determined by $H$.  Likewise, with
% \begin{equation}\label{Qdef}
%  Q(u):=\rho(V')^2,
% \end{equation}
% one obtains by integrating the second equation
% \begin{equation}\label{Qformula}
%  Q(u)=\frac13-H^2+\frac23H^3+k_2uH^2
%       -k_2\int_0^uH(s)^2\,\dd s.
% \end{equation}
% Hence $Q$ is also determined by $H$, while
% \begin{equation}\label{QPH}
%  Q=\rho P(H')^2.
% \end{equation}
% At any interior point, $P>0$, $Q>0$ and $H'\ne0$, so \eqref{QPH} uniquely determines $\rho$.  Therefore $\rho_1=\rho_2$, a contradiction.
If \(D\equiv0\), write the common phase graph as \(H\). For
\(i=1,2\), let
\[
\xi_i(u):=U_i^{-1}(u),
\qquad 0<u<1,
\]
and define
\begin{equation}\label{Pdef}
P_i(u)
:=
\bigl[U_i'(\xi_i(u))\bigr]^2.
\end{equation}
Using the first equation of \eqref{SW}, we obtain
\begin{equation}\label{Pformula}
P_i(u)
=
-2\int_0^u
s\bigl(1-s-k_1H(s)\bigr)\,\dd s.
\end{equation}
Hence \(P_1=P_2\); we denote their common value by \(P\).

Likewise, define
\begin{equation}\label{Qdef}
Q_i(u)
:=
\rho_i
\bigl[V_i'(\xi_i(u))\bigr]^2.
\end{equation}
Integrating the second equation of \eqref{SW} gives
\begin{equation}\label{Qformula}
Q_i(u)
=
\frac{1}{3}
-H(u)^2
+\frac{2}{3}H(u)^3
+k_2uH(u)^2
-k_2\int_0^u H(s)^2\,\dd s.
\end{equation}
Thus \(Q_1=Q_2\); denote their common value by \(Q\).

Since
\[
V_i'(\xi_i(u))
=
H'(u)U_i'(\xi_i(u)),
\]
we have, for \(i=1,2\),
\begin{equation}\label{PQ_relation}
Q(u)
=
\rho_iP(u)\bigl(H'(u)\bigr)^2.
\end{equation}
At every interior point \(u\in(0,1)\),
\[
P(u)>0,\qquad
Q(u)>0,\qquad
H'(u)\neq0.
\]
Consequently, \eqref{PQ_relation} uniquely determines \(\rho_i\), and
therefore \(\rho_1=\rho_2\), a contradiction.
\end{proof}

Tangential zeros of \(D\) that do not change sign do not affect the argument. List the sign-changing zeros as
\begin{equation}\label{zlist}
 z_1<z_2<\cdots<z_{2m+1}.
\end{equation}
The number of sign-changing zeros is odd by \eqref{endsigns}. The case \(m=0\) already contradicts the two moment identities and is also covered by the argument below. Set
\begin{equation}\label{ecdef}
 e_0=0,\quad e_j=z_{2j}\ (1\le j\le m),\quad e_{m+1}=1,
 \qquad
 c_j=z_{2j+1}\ (0\le j\le m).
\end{equation}
Thus on every phase-plane lens $(e_j,e_{j+1})$,
\begin{equation}\label{lenssign}
 % D<0\quad\text{on }(e_j,c_j),
 % \qquad
 % D>0\quad\text{on }(c_j,e_{j+1}).
D(u)\leq0
\quad\text{for }u\in(e_j,c_j),
\qquad
D(u)\geq0
\quad\text{for }u\in(c_j,e_{j+1}).
\end{equation}

Define the cumulative moments
\begin{equation}\label{ABdef}
 A(u):=\int_0^u sD(s)\,\dd s,
 \qquad
 B(u):=\int_0^u G(s)D(s)\,\dd s.
\end{equation}
Then by \eqref{Dmoments},
\begin{equation}\label{ABends}
 A(0)=A(1)=B(0)=B(1)=0.
\end{equation}
Finally set
\begin{equation}\label{qdef}
 q(u):=\frac{G(u)}{u},\qquad 0<u<1.
\end{equation}
Since $G>0$ and $G'<0$,
\begin{equation}\label{qdec}
 q'(u)=\frac{uG'(u)-G(u)}{u^2}<0.
\end{equation}

\begin{lemma}\label{lem:multilens}
There exists a downward sign-changing crossing $u_*=z_{2j_*}$ such that
\begin{equation}\label{ABpositive}
 A(u_*)>0,
 \qquad
 B(u_*)>0.
\end{equation}
\end{lemma}

\begin{proof}
Write
\[
 A_j:=A(e_j),\qquad B_j:=B(e_j),\qquad q_j:=q(c_j).
\]
On the $j$th lens, \eqref{lenssign} and the strict decrease of $q$ give
\begin{align}
 &(B_{j+1}-B_j)-q_j(A_{j+1}-A_j)\notag\\
 &\quad=\int_{e_j}^{e_{j+1}}u\,[q(u)-q_j]D(u)\,\dd u<0.
 \label{lensineq}
\end{align}
Indeed, the integrand is strictly negative on both sides of $c_j$ except at isolated zeros.

Summing \eqref{lensineq} for $j=0,\dots,m$ and using \eqref{ABends} yields
\[
 \sum_{j=0}^m q_j(A_{j+1}-A_j)>0.
\]
Discrete summation by parts gives
\begin{equation}\label{summationA}
 \sum_{j=1}^m(q_{j-1}-q_j)A_j>0.
\end{equation}
If \(m=0\), then the left-hand side of the preceding inequality is
the empty sum and hence equals zero, contradicting its strict
positivity. Therefore \(m\geq1\).

% Since $q_{j-1}-q_j>0$, some $A_j$ is positive.  
Since \(q_{j-1}-q_j>0\), there exists at least one index
\(j\in\{1,\ldots,m\}\) such that \(A_j>0\).
Let $j_*$ be the largest index with $A_{j_*}>0$.  Then $A_j\le0$ for $j>j_*$.  Summing \eqref{lensineq} from $j=j_*$ to $m$ gives
\begin{align*}
 -B_{j_*}
 &<\sum_{j=j_*}^m q_j(A_{j+1}-A_j)\\
 &=-q_{j_*}A_{j_*}
   +\sum_{j=j_*+1}^m(q_{j-1}-q_j)A_j<0.
\end{align*}
Hence $B_{j_*}>0$.  The point $e_{j_*}=z_{2j_*}$ is a downward crossing by construction, and \eqref{ABpositive} follows.
\end{proof}

\medskip
We can now complete the proof of Theorem \ref{thm:rigidity}.

\begin{proof}[Proof of Theorem \ref{thm:rigidity}]
Let $u_*$ be the downward crossing from Lemma \ref{lem:multilens}.  For each phase curve \(H_i\), define $P_i,Q_i$ by \eqref{Pformula} and \eqref{Qformula}.  
Subtracting the two identities \eqref{Pformula} at \(u_*\) gives
\begin{equation}\label{Porder}
 % P_2(u_*)-P_1(u_*)=2k_1A(u_*)>0.
\begin{aligned}
P_2(u_*)-P_1(u_*)
&=
2k_1\int_0^{u_*}
s\bigl(H_2(s)-H_1(s)\bigr)\,\dd s\\
&=
2k_1A(u_*)>0.
\end{aligned}
\end{equation}

Since \(H_1(u_*)=H_2(u_*)\), the local terms in \eqref{Qformula} are canceled, and hence
\begin{equation}\label{Qorder}
 % Q_2(u_*)-Q_1(u_*)=-k_2B(u_*)<0.
\begin{aligned}
Q_2(u_*)-Q_1(u_*)
&=
-k_2\int_0^{u_*}
\bigl(H_1(s)+H_2(s)\bigr)
\bigl(H_2(s)-H_1(s)\bigr)\,\dd s\\
&=
-k_2B(u_*)<0.
\end{aligned}
\end{equation}
Using $Q_i=\rho_iP_i(H_i')^2$, we obtain
\begin{equation}\label{sloperatio}
 \frac{H_2'(u_*)^2}{H_1'(u_*)^2}
 =\frac{Q_2(u_*)}{Q_1(u_*)}
  \frac{\rho_1}{\rho_2}
  \frac{P_1(u_*)}{P_2(u_*)}<1.
\end{equation}
Since
$H_1'(u_*)<0,
H_2'(u_*)<0,$
% $H_i'$ are negative; hence 
\eqref{sloperatio} implies
\[
 H_2'(u_*)>H_1'(u_*),
\]
and hence % or
\begin{equation}\label{Dprimepos}
 D'(u_*)>0.
\end{equation}
This is impossible because \(u_*\) is a downward sign-changing zero of \(D\), and therefore \(D'(u_*)\leq0\), with equality possible only for a higher odd-order crossing. This contradiction proves the theorem.
\end{proof}

\section{Characterization of the zero-speed threshold}\label{sec:global}

We return to the full traveling wave problem \eqref{TW} and collect the standard properties needed below.

\begin{proposition}\label{prop:standard}
In the strong-competition regime, the monotone bistable front is unique up to translation, and the corresponding propagation speed
\[
c=c(d,r,k_1,k_2)
\]
depends continuously on 
\[
(d,r,k_1,k_2)\in(0,\infty)^2\times(1,\infty)^2.
\]
More precisely, fix the translation by imposing
\[
U(0)=\frac{1}{2}.
\]
Suppose that
\[
(d_n,r_n,k_{1,n},k_{2,n})
\longrightarrow
(d,r,k_1,k_2)
\quad
\text{in}
\quad
(0,\infty)^2\times(1,\infty)^2.
\]
If
\[
(c_n,U_n,V_n)
\]
denotes the corresponding normalized monotone front, then
\[
c_n\longrightarrow c
\]
and
\[
(U_n,V_n)
\longrightarrow
(U,V)
\qquad\text{in }
C^2_{\mathrm{loc}}(\mathbb R)
\times
C^2_{\mathrm{loc}}(\mathbb R).
\]
% convergence of the coefficients implies local \(C^2\)-convergence of
% the corresponding wave profiles and convergence of their speeds. 

For fixed $d,r,k_2$, the speed is strictly decreasing in $k_1$.
Moreover, there exists a unique threshold $k_*(d,r,k_2)>1$ 
% satisfying \eqref{thresholdknown}.
such that
\[
c(d,r,k_1,k_2)
\begin{cases}
>0, & 1<k_1<k_*,\\
=0, & k_1=k_*,\\
<0, & k_1>k_*.
\end{cases}
\]
\end{proposition}

\noindent
% The strict $k_1$-monotonicity and threshold statement are recorded explicitly in \cite{Kanon1995,NakamuraOgiwara2026}.  Continuity follows from the standard heteroclinic continuation/compactness argument for the unique monotone front; see \cite{Kanon1995,Volpert1994}.  Notice that this proposition uses no monotonicity with respect to $d$ or $r$.
The uniqueness, strict \(k_1\)-monotonicity, and threshold property
are classical(see \mbox{%DIFAUXCMD
\cite{Gardner1982,Kanon1995,NakamuraOgiwara2026}}\hskip0pt%DIFAUXCMD
).
The continuous dependence used above follows from the standard continuation and compactness theory for monotone heteroclinic fronts (see also \cite{Kanon1995,Volpert1994}).

Since the standing wave system depends on \((d,r)\) only through \(\rho=d/r\), Theorem \ref{thm:rigidity} has the following immediate global consequence.

\begin{lemma}\label{lem:injective}
Fix $k_2>1$ and define
\begin{equation}\label{Kdef}
 K_{k_2}(\rho):=k_*(\rho,1,k_2),\qquad \rho>0.
\end{equation}
Then $K_{k_2}$ is continuous and injective.
\end{lemma}

% \begin{proof}
% Continuity follows from Proposition \ref{prop:standard} and strict monotonicity of $c$ in $k_1$ (equivalently, from the elementary continuous-root lemma for a continuous function strictly monotone in its root variable).  If
% \[
%  K_{k_2}(\rho_1)=K_{k_2}(\rho_2)=k_1,
% \]
% then both $\rho_1$ and $\rho_2$ produce standing fronts for the same $k_1,k_2$.  Theorem \ref{thm:rigidity} implies $\rho_1=\rho_2$.
% \end{proof}

\begin{proof}
We first prove continuity. Let
\[
\rho_n\longrightarrow\rho>0
\]
and set
\[
\kappa:=K_{k_2}(\rho).
\]
Fix \(\varepsilon>0\) sufficiently small so that
\[
\kappa-\varepsilon>1.
\]
Since the propagation speed is strictly decreasing with respect to
\(k_1\), the defining property of \(\kappa\) yields
\[
c(\rho,1,\kappa-\varepsilon,k_2)>0
\]
and
\[
c(\rho,1,\kappa+\varepsilon,k_2)<0.
\]
By Proposition~\ref{prop:standard}, for all sufficiently large
\(n\),
\[
c(\rho_n,1,\kappa-\varepsilon,k_2)>0
\]
and
\[
c(\rho_n,1,\kappa+\varepsilon,k_2)<0.
\]
The defining property of \(K_{k_2}(\rho_n)\) therefore implies
\[
\kappa-\varepsilon
<
K_{k_2}(\rho_n)
<
\kappa+\varepsilon.
\]
Since \(\varepsilon>0\) is arbitrary, we conclude that
\[
K_{k_2}(\rho_n)
\longrightarrow
K_{k_2}(\rho).
\]
Thus \(K_{k_2}\) is continuous.

To prove injectivity, suppose that
\[
K_{k_2}(\rho_1)
=
K_{k_2}(\rho_2)
=
k_1.
\]
Then both \(\rho_1\) and \(\rho_2\) support monotone standing fronts
for the same competition coefficients \(k_1,k_2\).
Theorem~\ref{thm:rigidity} gives
\[
\rho_1=\rho_2.
\]
Hence \(K_{k_2}\) is injective.
\end{proof}

A continuous injective real-valued function on an interval is strictly monotone. The Maxwell identity determines the direction of monotonicity.
\begin{lemma}\label{lem:decrease}
For every $k_2>1$, $K_{k_2}$ is strictly decreasing and
\begin{equation}\label{K1}
 K_{k_2}(1)=k_2.
\end{equation}
Moreover,
\begin{equation}\label{sideb}
 K_{k_2}(\rho)>k_2\quad(0<\rho<1),
 \qquad
 K_{k_2}(\rho)<k_2\quad(\rho>1).
\end{equation}
\end{lemma}

% \begin{proof}
% Apply Theorem \ref{thm:Maxwell} to the standing front with $k_1=K_{k_2}(\rho)$.  Since $\E>0$,
% \[
%  (\rho-1)(K_{k_2}(\rho)-k_2)<0
% \]
% when $\rho\ne1$.  At $\rho=1$, \eqref{Maxwell} forces $K_{k_2}(1)=k_2$.  Therefore \eqref{sideb} holds.  Since $K_{k_2}$ is continuous and injective by Lemma \ref{lem:injective}, it is strictly monotone; \eqref{sideb} excludes strict increase.
% \end{proof}

\begin{proof}
Apply Theorem~\ref{thm:Maxwell} to the standing front with
\[
k_1=K_{k_2}(\rho).
\]
If \(\rho\neq1\), then \(\E>0\), and the positivity of
\[
\rho\bigl(K_{k_2}(\rho)-1\bigr)+(k_2-1)
\]
in \eqref{Maxwell} imply
\[
(\rho-1)\bigl(K_{k_2}(\rho)-k_2\bigr)<0.
\]
Therefore,
\[
K_{k_2}(\rho)>k_2
\quad\text{for }0<\rho<1,
\]
and
\[
K_{k_2}(\rho)<k_2
\quad\text{for }\rho>1.
\]

At \(\rho=1\), identity \eqref{Maxwell} reduces to
\[
\bigl(K_{k_2}(1)-k_2\bigr)
\bigl(K_{k_2}(1)+k_2-2\bigr)=0.
\]
Since \(K_{k_2}(1)>1\) and \(k_2>1\), the second factor is positive.
Hence
\[
K_{k_2}(1)=k_2.
\]

By Lemma~\ref{lem:injective}, \(K_{k_2}\) is continuous and
injective, and is therefore strictly monotone. The inequalities above
exclude strict increase, so \(K_{k_2}\) is strictly decreasing.
\end{proof}

\subsection{The two endpoint limits}

% The small-$\rho$ endpoint is supplied by the near-degenerate theory.  Alzahrani--Davidson--Dodds prove that for fixed $k_1,k_2>1$ and sufficiently small diffusion ratio, the sign is determined by the degenerate balance $k_1=k_2^2$; in the notation used here,
% \begin{equation}\label{smalldsign}
%  \rho\ll1:
%  \qquad
%  k_1>k_2^2\Rightarrow c<0,
%  \qquad
%  k_1<k_2^2\Rightarrow c>0.
% \end{equation}
% See \cite{AlzahraniDavidsonDodds2010} and the discussion in Section 5 of \cite{NakamuraOgiwara2026}.

The limit as \(\rho\downarrow0\) follows from the near-degenerate theory of Alzahrani--Davidson--Dodds \cite{AlzahraniDavidsonDodds2010}. More precisely, for fixed \(k_1,k_2>1\) with \(k_1\neq k_2^2\), there exists
\[
\rho_0=\rho_0(k_1,k_2)>0
\]
such that, for every \(0<\rho<\rho_0\),
\begin{equation}\label{smalldsign}
\begin{cases}
c(\rho,1,k_1,k_2)<0,
&
k_1>k_2^2,
\\[1mm]
c(\rho,1,k_1,k_2)>0,
&
1<k_1<k_2^2.
\end{cases}
\end{equation}
See also the discussion in Section~5 of
\cite{NakamuraOgiwara2026}.

\begin{lemma}[Small-ratio limit]\label{lem:smallend}
For every $k_2>1$,
\begin{equation}\label{smallend}
 \lim_{\rho\downarrow0}K_{k_2}(\rho)=k_2^2.
\end{equation}
\end{lemma}

\begin{proof}
Since $K_{k_2}$ is decreasing, the limit $L_0:=\lim_{\rho\downarrow0}K_{k_2}(\rho)$ exists in $(k_2,\infty]$.  
If $k_1>k_2^2$, \eqref{smalldsign} and \eqref{signK} imply $k_1>K_{k_2}(\rho)$ for all sufficiently small $\rho$, hence $L_0\le k_1$.  
Letting $k_1\downarrow k_2^2$ gives $L_0\le k_2^2$.  If $1<k_1<k_2^2$, the positive-speed half of \eqref{smalldsign} gives $k_1<K_{k_2}(\rho)$ for small $\rho$, hence $L_0\ge k_1$.  Letting $k_1\uparrow k_2^2$ yields $L_0\ge k_2^2$.
\end{proof}

The opposite endpoint is obtained from species exchange. The propagation speed satisfies \cite{NakamuraOgiwara2026}
\begin{equation}\label{exchange}
 c(d,r,k_1,k_2)=-\sqrt{dr}\;c(1/d,1/r,k_2,k_1).
\end{equation}

\begin{lemma}[Large-ratio limit]\label{lem:largeend}
For every $k_2>1$,
\begin{equation}\label{largeend}
 \lim_{\rho\to\infty}K_{k_2}(\rho)=\sqrt k_2.
\end{equation}
\end{lemma}

\begin{proof}
Since \(K_{k_2}\) is strictly decreasing and bounded below by \(1\),
the limit
$
L_\infty
:=
\lim_{\rho\to\infty}K_{k_2}(\rho)
$
exists and satisfies
$
1\leq L_\infty<k_2.
$

Suppose first that
$
L_\infty>\sqrt{k_2}.
$
Choose \(k_1\) such that
\[
\sqrt{k_2}<k_1<L_\infty.
\]
For all sufficiently large \(\rho\),
$
k_1<K_{k_2}(\rho),
$
and hence
\[
c(\rho,1,k_1,k_2)>0.
\]
Using the species-exchange identity \eqref{exchange}, we
obtain
\[
c(1/\rho,1,k_2,k_1)<0.
\]
However,
$
k_2<k_1^2.
$
Applying the small-ratio sign rule \eqref{smalldsign} to the
competition pair \((k_2,k_1)\) gives
\[
c(1/\rho,1,k_2,k_1)>0
\]
for all sufficiently large \(\rho\), a contradiction.

Suppose next that
$
L_\infty<\sqrt{k_2}.
$
Choose \(k_1\) such that
$
L_\infty<k_1<\sqrt{k_2}.
$
For all sufficiently large \(\rho\),
$
k_1>K_{k_2}(\rho),
$
and therefore
\[
c(\rho,1,k_1,k_2)<0.
\]
The species-exchange identity then yields
\[
c(1/\rho,1,k_2,k_1)>0.
\]
On the other hand,
$
k_2>k_1^2,
$
so the small-ratio sign rule applied to the pair \((k_2,k_1)\) gives
\[
c(1/\rho,1,k_2,k_1)<0
\]
for all sufficiently large \(\rho\), again a contradiction.
Consequently,
$
L_\infty=\sqrt{k_2}.
$
\end{proof}

% \begin{proof}
% By monotonicity, $L_\infty:=\lim_{\rho\to\infty}K_{k_2}(\rho)$ exists and $1\le L_\infty<k_2$.  Suppose first that $L_\infty>\sqrt k_2$ and choose
% \[
%  \sqrt k_2<k_1<L_\infty.
% \]
% For all sufficiently large $\rho$, $k_1<K_{k_2}(\rho)$, hence
% \[
%  c(\rho,1,k_1,k_2)>0.
% \]
% Using \eqref{exchange},
% \[
%  c(1/\rho,1,k_2,k_1)<0.
% \]
% But $k_2<k_1^2$, so the small-ratio rule \eqref{smalldsign}, applied with the competition pair $(k_2,k_1)$, gives $c(1/\rho,1,k_2,k_1)>0$ when $\rho$ is large, a contradiction.

% If $L_\infty<\sqrt k_2$, choose $L_\infty<a<\sqrt k_2$.  Then $k_1>K_{k_2}(\rho)$ for large $\rho$, so $c(\rho,1,k_1,k_2)<0$ and, by \eqref{exchange}, $c(1/\rho,1,k_2,k_1)>0$.  Since now $k_2>a^2$, the small-ratio rule gives the opposite sign for large $\rho$, again a contradiction.  Thus $L_\infty=\sqrt k_2$.
% \end{proof}

\begin{proof}[Proof of Theorem \ref{thm:main}]
Existence and uniqueness of the threshold \(K_{k_2}(\rho)\) for each \(\rho>0\) follow from Proposition \ref{prop:standard}. Continuity and injectivity are given by Lemma \ref{lem:injective}, while Lemma \ref{lem:decrease} yields strict decrease. The central value and the two endpoint limits are \eqref{K1}, \eqref{smallend}, and \eqref{largeend}. Finally, the classification \eqref{signK} is exactly the strict \(k_1\)-monotonicity in \eqref{thresholdknown}, evaluated relative to the zero-speed threshold. Hence the range of \(K_{k_2}\) is \((\sqrt{k_2},k_2^2)\).
\end{proof}

\begin{proof}[Proof of Corollaries \ref{cor:sharp} and \ref{cor:symmetric}]
Corollary \ref{cor:sharp} follows immediately from the strict decrease and range of $K_{k_2}$ together with \eqref{signK}.  
If $k_1=k_2=k$, then $K_k(1)=k$.
The injectivity of $K_k$ implies 
\[
k=K_k(\rho)
\iff
\rho=1.
\]
Since $K_k$ is strictly decreasing,
\[
\rho<1
\implies
k<K_k(\rho)
\implies
c(d,r,k,k)>0,
\]
whereas
\[
\rho>1
\implies
k>K_k(\rho)
\implies
c(d,r,k,k)<0.
\]
Since \(\rho=d/r\), this proves  \eqref{symzero}--\eqref{symsign}.
\end{proof}

\begin{remark}
The main result of this paper concerns the structure of the zero-speed set rather than 
pointwise monotonicity of the full propagation speed. More precisely, we prove that, for 
each fixed $k_2>1$, the zero-speed set is described by a strictly decreasing threshold
\[
    k_1=K_{k_2}\left(\frac{d}{r}\right).
\]
We do not claim that $c(d,r,k_1,k_2)$ is monotone in $d$ or in $r$ separately. Moreover, 
the strict monotonicity of $K_{k_2}$ does not by itself imply the pointwise inequality
\[
    K_{k_2}'(\rho)<0
    \qquad\text{for every }\rho>0.
\]
Such a differential statement would require additional information on the dependence of 
the standing profile on $\rho$.
\end{remark}

% \section*{Acknowledgments}

% S. Ma is partially supported by the National Natural Science Foundation of China (No. 12501686) and the China Postdoctoral Science Foundation (2024M751949); 
% %\textcolor{blue}{D. Xiao is partially supported by the National Natural Science Foundation of China (No. XXXX);}
% M. Zhou is partially supported by
% the National Key
% Research and Development Program of China (2021YFA1002400), Scientific Research
% Innovation Capability Support Project for Young Faculty (SRICSPYF-ZY2025172), National Natural Science Foundation of China (No. 12271437, 11971498), the Fundamental Research Funds for the Central Universities (Nankai University 63241642), Tianjin Natural Science Foundation Outstanding
% Youth Project (23JCJQJC00190) and Nankai Zhide Foundation.

\end{document}